\documentclass{amsart}
\usepackage{graphicx}
\usepackage{amssymb,amscd}
\usepackage{color}
\usepackage{enumitem}
\usepackage{bm}

 \newtheorem{thm}{Theorem}[section]
 \newtheorem{thmA}{Theorem}
 \newtheorem{cor}[thm]{Corollary}
 \newtheorem{lemma}[thm]{Lemma}
 \newtheorem{prop}[thm]{Proposition}
 \theoremstyle{definition}
 \newtheorem{defn}[thm]{Definition}
 
 \newtheorem{quest}[thm]{Question}

 \theoremstyle{remark}
 \newtheorem{rem}[thm]{Remark}

 \numberwithin{equation}{subsection}

\newcommand{\FF}{\text{$\mathcal{F}$}}
\newcommand{\GG}{\text{$\mathcal{G}$}}

\newcommand{\str}{\operatorname{sat}}

        \newcommand{\field}[1]{\text{$\mathbb{#1}$}}
        \newcommand{\N}{\field{N}}
        \newcommand{\Z}{\field{Z}}
        
        \newcommand{\R}{\field{R}}

\newdimen\theight
\def\TeXref#1{%
             \leavevmode\vadjust{\setbox0=\hbox{{\tt
                     \quad\quad  {\small \textrm #1}}}%
             \theight=\ht0
             \advance\theight by \lineskip
             \kern -\theight \vbox to
             \theight{\rightline{\rlap{\box0}}%
             \vss}%
             }}%

\begin{document}

\title{Nonperiodic leaves of codimension one foliations}

\author{Carlos Meni\~no Cot\'on}

\address{$\ddagger$ CITMAGA, 15782 Santiago de Compostela, Spain; and Departamento de Matemática Aplicada 1, Universidad de Vigo, Escola de Enxe\~ner\'ia Industrial, Rua Conde de Torrecedeira 86, CP 36208, Vigo, Spain.}

\email{carlos.menino@uvigo.es}

\thanks{The author wants to thank to the the Ministerio de Ciencia e Innovación (Spain) grant
PID2020-114474GB-I00, the MathAmSud 2019-2020 CAPES-Brazil (``Rigidity and Geometric Structures on Dynamics'') and the CNPq research grant 310915/2019-8 that partially supported this research.}

\begin{abstract}
It is shown an example of a noncompact and nonperiodic $5$-manifold that is not homeomorphic to any leaf of any $C^2$ codimension one foliation of any compact $6$-manifold, but it is homeomorphic to a (proper) leaf of a $C^1$ codimension one foliation and also to a (proper) leaf of a $C^\infty$ codimension $2$ foliation. As far as we know, this is the first example of this nature. In addition, it is shown examples of codimension one $C^{r}$ foliations, $r\in[0,2)$, with a minimal invariant set whose leaves are pairwise nonhomeomorphic.
\end{abstract}

\maketitle

\section{Introduction}

In \cite{Sondow}, J. Sondow posed the question about what manifolds can be realized as leaves of foliation in a compact manifold, this is the so called ``Realization Problem''. The realization problem is in fact a list of problems that depend in two choices:
\begin{itemize}
\item[(1)] How the manifold must be realized as a leaf: up to homeomorphism, up to diffeomorphism, up to quasi-isometry, etc.

\item[(2)] In what kind of foliations we want to study the problem (e.g.: codimension one, $C^2$ transverse regularity, Riemannian, etc.)
\end{itemize}

In the present work we are interested in the realization problem up to homeomorphism in codimension one foliations on compact manifolds. The notation $C^k$ will always refer to the transverse regularity (regularity of the holonomy maps) of the given foliation.

\begin{defn}[Nonleaf]
A manifold $Z$ is called a topological (resp. smooth, resp. metric) {\em nonleaf} if it is not homeomorphic (resp. diffeomorphic, resp. quasi-isometric) to any leaf of any codimension one foliation on any (Riemannian) compact manifold.

If regularity matters, we shall say that $Z$ is a $C^k$ nonleaf if it is not homeomorphic to any leaf of any transversely $C^k$ foliation on any compact manifold.
\end{defn}

Several examples of nonleaves were described since the question was formulated, the following is a nonexhaustive list of references:

\begin{itemize}
\item Topological nonleaves of codimension one foliations were firstly obtained by E. Ghys in \cite{Ghys1985} for transverse regularity $C^0$. See also \cite{JAP1985} for a similar example\footnote{In that work it is shown a family of $C^2$ nonleaves although it can be shown that lots of these examples are actually $C^0$ nonleaves, see \cite{Menino-Schweitzer2022}} and \cite{Attie-Hurder1996} for nonleaves which are homotopic to leaves.

\item Smooth nonleaves that are homeomorphic to leaves were described in \cite{Menino-Schweitzer2018, Menino-Schweitzer2022}.

\item Quasi-isometric nonleaves were first described in \cite{Janusz}, and more examples (working for in any codimension!) were also given in \cite{Attie-Hurder1996,Zeghib}.
\end{itemize}

All these examples share a common point, they are (topologically, smoothly or coarse-geometrically) nonperiodic at some of their ends. It is a natural question what kind of nonperiodic manifolds can be realized as leaves of codimension one foliations. Moreover, it is unknown if a manifold that is not homeomorphic to a leaf of any ($C^0$) codimension one foliation can be homeomorphic to a leaf of some foliation of higher codimension. This question is even more difficult if we also require that the manifold must be homeomorphic to a proper leaf, i.e., a leaf that is not contained in its closure.

We give a positive answer to this question for transversely $C^2$ codimension one foliations, more precisely:

\begin{thmA}
For every $0\leq r<2$, there exists a $5$-manifold $Z_r$ which is not homeomorphic to any leaf of any transversely $C^2$ codimension one foliation on any compact manifold but it is homeomorphic to a proper leaf of a $C^r$ codimension one foliation on a compact manifold and to a proper leaf of a $C^\infty$ codimension two foliation on a compact manifold.
\label{Thm1}\end{thmA}

These manifolds $Z_r$ are very similar to the manifolds $M_\omega$ introduced by E. Ghys in \cite[pp. 396--399]{Ghys1995} whose topology is encoded by a bi-infinite sequence $\omega:\Z\to\{0,1\}$. It is worth to point out that these $C^r$ realizations are only possible for some almost periodic choices of $\omega$. This leads to the question of whether ``far from periodic'' Ghys manifolds can be realized as proper leaves of some codimension one foliation.

The manifolds $M_\omega$ were originally used to exhibit a codimension two foliation on a compact manifold with an exceptional minimal set such that every pair of leaves in the minimal set are pairwise not homeomorphic. Our construction reproduces this behavior for given codimension one $C^r$ foliations, $r\in[0,2)$, this solves in the positive an open question of E. Ghys in \cite[pp. 399]{Ghys1995}.

These realizations were inspired by a construction given by J. Cantwell and L. Conlon in \cite[Section 9]{Cantwell-Conlon2001} (see also \cite[Theorem 1.4.8]{Walczak2004}), in that work they were only interested in the realization of the dynamics of some special cases of pseudogroups (Markov pseudogroups) on minimal sets of codimension one foliations. We shall exploit this construction in several ways in order to find our desired foliations and realizations.

\section{Ghys manifolds and nonleaves}

Let $M_0,M_1$ be two closed $n$-manifold, $n>3$, such that $\pi_1(M_0)=\Z^{m_0}$ and $\pi_1(M_1)=\Z^{m_1}$ with $m_0\neq m_1$ and $m_0,m_1>2$. Let $\alpha_i,\beta_i$ be two loops whose homotopy classes are generators of $\pi_1(M_i)$, $i=0,1$.

Let $2\beta_i$ be a simple loop homotopic to $\beta_i\ast\beta_i$ and let $N(\alpha_i)$ and $N(2\beta_i)$ be pairwise disjoint tubular neighborhoods of these loops in $M_i$, $i=0,1$.

Set $V_i=M_i\setminus\left(N(\alpha_i)\cup N(2\beta_i)\right)$, $i=1,2$. These are compact $n$-manifolds with boundary homeomorphic to two copies of $S^{n-2}\times S^1$, one component will be denoted by $T_{\alpha_i}$ and the other by $T_{2\beta_i}$ in corresponde with the removed tubular neighborhood bounded by them. Since $n>3$, it follows that $\pi_1(V_i)=\pi_1(M_i)$, $i=0,1$. These manifolds $V_0,V_1$ will be called {\em fundamental blocks}.

Let $\Omega=\{0,1\}^\Z$ and let $\omega\in \Omega$, i.e. a bi-sequence $\omega:\Z\to\{0,1\}$. The bilateral shift $\sigma:\Omega\to \Omega$ is defined as $\sigma(\omega)(i)=\omega(i+1)$. Let $h_{ij}: T_{2\beta_i}\to T_{\alpha_j}$ be a homeomorphism mapping the meridian in the direction of $2\beta_i$ to the meridian in the direction of $\alpha_j$, $i,j\in\{0,1\}$.

\begin{defn}[Ghys manifold]
Let $V_0,V_1$ be two fundamental blocks as above and let $\omega\in \Omega$. The manifold $M_\omega$ defined as
\[
\cdots\cup_\partial V_{\omega(-1)}\cup_\partial V_{\omega(0)}\cup_\partial V_{\omega(1)}\cup_\partial\cdots
\]
where the boundary unions between the blocks $V_{\omega(i)}$ and $V_{\omega(i+1)}$ are performed just between the boundary components $T_{2\beta_{\omega(i)}}$ and $T_{\alpha_{\omega(i+1)}}$  via the homeomorphism $h_{\omega(i),\omega(i+1)}$.

This manifold has no boundary, is noncompact and has two ends. The manifold $M_\omega$ will be called a {\em Ghys manifold} and $\omega$ will be called its {\em associated sequence}.
\end{defn}

Our interest on Ghys manifolds is the fact that their topologies are determined by the sequence $\omega$ (or more precisely, by the $\sigma$-orbit of $\omega$). In \cite[pp.396--399]{Ghys1995}, E. Ghys only consider the case where $n=4$, but his work holds in any dimension above $4$, showing that $M_\omega$ is homeomorphic to $M_{\omega'}$ if and only if there exists $k\in\Z$ such that $\sigma^k(\omega)=\omega'$.

\begin{defn}[Periodic end]
Let $\bm{e}$ be an end of a (noncompact) manifold $M$. It is said that $\bm{e}$ is periodic if there exists a neighborhood $N_{\bm{e}}$ for that end and an embedding $\iota:N\hookrightarrow N$ such that the family $\{\iota^n(N)\}_{n\in\N}$ is a neighborhood system for $\bm{e}$. The manifold $\overline{N\setminus \iota(N)}$ is called a {\em periodic segment}.
\end{defn}

\begin{defn}[End periodic sequences]
Let $\mathfrak{a}=(a_n)_{n\in\N}$ be a sequence. It is said that $\mathfrak{a}$ is {\em end periodic} if there exists $N,p\in\N$ such that $a_{n+p}=a_n$ for all $n\geq N$.

Let $\omega\in\{0,1\}^\Z$. It is said that $\omega$ is {\em forward} (resp. {\em backward}) periodic if $(\omega(n))_{n\geq 0})$ (resp. $(\omega(-n))_{n\geq 0}$) is end periodic .
\end{defn}

\begin{prop}\label{p:nonperiodic_Ghys}
The ends of a Ghys manifold $M_\omega$ are periodic if and only if $\omega$ is forward and backward periodic. 
\end{prop}

Proposition~\ref{p:nonperiodic_Ghys} is not a trivial consequence of the work given in \cite[pp. 398]{Ghys1995}. The subtle point is the following: it is not clear that a Ghys manifold associated to a forward or backward nonperiodic  sequence $\omega$ is not homeomorphic to some (topologically) end periodic manifold.

In order to show this, we shall need to recall some of the properties of the fundamental group of a Ghys manifold.

Set $\Gamma^\omega_i = \Z^{m_0}$ if $\omega(i)=0$ and $\Gamma^\omega_i = \Z^{m_1}$ if $\omega(i)=1$. Let $a_i$ and $b_i$ be generators of $\Gamma^\omega_i$ which correspond with homotopy classes of $\alpha_{\omega(i)}$ and $\beta_{\omega(i)}$ in the fundamental block $V_{\omega(i)}$.

The fundamental group $\pi_1(M_\omega)=\Gamma^\omega$ is obtained as the amalgamated product of the $\Gamma^\omega_i$'s under the relations $a_{i+1} = 2b_{i}$.

We resume here the fundamental properties of this group, the proof can be found in \cite[pp. 396--399]{Ghys1995}:
\begin{enumerate}[label=\textbf{P.\arabic*}]
\item \label{i:Property1}Every abelian subgroup of $\Gamma^\omega$ with maximal rank $m_0$ (resp. $m_1$) is conjugated to one of the $\Gamma^\omega_i$'s with $\omega(i)=0$ (resp. $\omega(i)=1$).

\item \label{i:Property2}For $i\neq j$, let $G_i$ and $G_j$ subgroups of $\Gamma^\omega$ conjugated to $\Gamma^\omega_i$ and $\Gamma^\omega_j$ respectively. If $G_i\cap G_j$ is nontrivial then $|i-j| = 1$. Moreover, $j=i+1$ if there exists a generator of $G_i\cap G_j$ that is indivisible in $G_j$ but divisible by $2$ in $G_i$. 
\end{enumerate}

Let $\tau:\N\to\{0,1\}$ be a sequence. Set $$N^+_\tau=V_{\tau(0)}\cup_\partial V_{\tau(1)}\cup_\partial\cdots\,,\ N^-_\tau = \cdots\cup_\partial V_{\tau(-1)}\cup_\partial V_{\tau(0)}\;,$$ they model the ends of a Ghys manifold. the fundamental groups $\pi_1(N^\pm_\tau)$, denoted by $\Gamma^\tau_\pm$, are also defined by an amalgamated product by the same relations as the given above for $\Gamma^\omega$. The same arguments given in \cite{Ghys1995} show that properties \ref{i:Property1} and \ref{i:Property2} hold for the groups $\Gamma^\tau_\pm$.

\begin{rem}\label{r:inclusion}
Let $k\in \Z$ and $\omega\in\{0,1\}^\Z$. Set $\tau_\pm:\N\to\{0,1\}$ defined\footnote{In this work we will assume that $0\in\N$.} as $\tau_\pm(i)=\omega(k\pm i)$. We shall denote $\Gamma^{\tau_\pm}_\pm$ by $\Gamma^{\omega,k}_\pm$ in order to remark its relation with $\Gamma^\omega$. We shall also denote $N^+_{\tau_+}=N^+_{\omega,k}$ and $N^-_{\tau_-}=N^-_{\omega,k}$.

Observe that both $\Gamma^{\omega,k}_\pm$ are subgroups of $\Gamma^\omega$, this means that the inclusions of $N^\pm_{\omega,k}$ in $M_\omega$ induce monomorphisms at the level of the fundamental group.
\end{rem}

With these properties in mind we can prove the following Lemma.

\begin{lemma}\label{l:endnonperiodicty_Ghys}
Let $k\in\Z$ and $\omega\in\{0,1\}^\Z$. Assume that there exists a monomorphism $\jmath:\Gamma^{\omega,k}_+\to \Gamma^{\omega,k}_+$ (resp. $\jmath:\Gamma^{\omega,k}_-\to\Gamma^{\omega,k}_-$) such that each $\jmath(\Gamma^\omega_i)$, $i\geq k$ (resp. $i\leq k$), is an abelian subgroup of maximal rank in $\Gamma^{\omega,k}_+$ (resp. $\Gamma^{\omega,k}_-$) then there exists $p\in\N$ such that $\omega(i)=\omega(i+p)$ (resp. $\omega(i)=\omega(i-p)$), for all $i\geq k$.
\end{lemma}
\begin{proof}
By means of the property \ref{i:Property1} of $\Gamma^{\omega,k}_+$, it follows that each $\jmath(\Gamma^\omega_i)$, $i\geq k$, is conjugated to some $\Gamma^\omega_{\theta(i)}$ with $\omega(\theta(i))=\omega(i)$ and $\theta(i)\geq k$. Using now
the second property \ref{i:Property2}, it follows that $\theta(i+1)=\theta(i)+1$. Applying inductively these properties the result follows. The argument is analogous for $\Gamma^{\omega,k}_-$.
\end{proof}

\begin{proof}[Proof of Proposition~\ref{p:nonperiodic_Ghys}]
Assume that the forward end of $M_\omega$ is (topologically) periodic. Then there exists a neighborhood $N$ of the forward end of $M_\omega$ and an end-periodic embedding $\iota:N\hookrightarrow N$. Without loss of generality we can assume that $N = N^+_{\omega,k}$ for some $k\in\Z$ (just by restricting $\iota$ to such a neighborhood and replacing $\iota$ by a suitable iteration of itself). Its fundamental group is $\Gamma^{\omega,k}_+$ (as defined in Remark~\ref{r:inclusion}). 

Observe that $\iota^{-1}$, which a priori is only defined in $\iota(N)$, can be extended to an embedding of $N$ in $M_\omega$, this can be guaranteed if we choose $k$ sufficiently large. Suppose that $\gamma$ is a homotopically nontrivial loop in $N$ such that $\iota\circ\gamma$ is trivial, then, applying $\iota^{-1}_\ast$ we should obtain that $\gamma$ is trivial in the ambient manifold $M_\omega$. But this is not possible since every nontrivial loop in $N$ is also nontrivial in $M_\omega$ (by the choice of $N$). It follows that $\iota_\ast$ is injective. A similar argument yields that $\iota^{-1}_\ast:\pi_1(N)\to\pi_1(\iota^{-1}(N))$ is also a monomorphism.

On the other hand the groups $\iota_\ast(\Gamma^\omega_i)$, $i\geq k$, are abelian subgroups of $\Gamma^{\omega,k}_+$ of maximal rank. Otherwise, it would exist an element $x\in\Gamma^{\omega,k}_i$ commuting with $\iota_ \ast(\Gamma^\omega_i)$ but $x\notin \iota_\ast(\Gamma^\omega_i)$. However $\Gamma^\omega_i$ is of maximal rank in $\Gamma^{\omega,k}_+$ and therefore $\iota_\ast^{-1}(x)$ would belong to $\Gamma^{\omega,k}_i$, this would imply that $\iota_\ast^{-1}$ is not a monomorphism, a contradiction.

Choose $\ell\in\N$ large enough so that $\iota^\ell(N)\subset N^+_{\omega, k+1}$. The morphism $\iota^\ell_\ast:\pi_1(N)\to\pi_1(N)$ is injective by the previous reasoning and preserves abelian subgroups of maximal rank. It follows, by Lemma~\ref{l:endnonperiodicty_Ghys}, that there exists $p\in\N$ such that $\omega(i)=\omega(i+p)$ for all $i\geq k$. Observe that $p\neq 0$ since $\iota^\ell(V_\omega(k))\subset N^+_{\omega,k+1}$ and therefore $\iota^\ell_\ast(\Gamma^\omega_k)$ must be conjugated with some $\Gamma^\omega_p$ in $\Gamma^{\omega,k+1}_+$. Thus the sequence $\omega$ is forward periodic as desired.

An analogous reasoning holds for topologically backward periodic Ghys manifolds, and this completes the proof.
\end{proof}


We can see Ghys manifolds associated to nonperiodic sequences $\omega:\Z\to\{0,1\}$ as the simplest examples of topologically nonperiodic manifolds with finitely many ends. In the same work \cite[pp. 396--399]{Ghys1995}, E. Ghys shows that any $M_\omega$ can be realized as a leaf of a $C^\infty$ codimension $2$ manifold (in fact, all of them can be realized simultaneusly in the same codimension $2$ foliation). However it was unclear how to realize such manifolds as a (proper) leaf in a codimension $1$ foliation (in fact, this is an implicit question in that work).

Moreover, a simple modification of a nonperiodic Ghys manifold can produce manifolds that cannot be realized as a leaf in any $C^2$ codimension one foliation on a compact manifold. In order to see this we shall use the following criterion for $C^2$ nonleaves given in \cite[Theorem 3.5]{Menino-Schweitzer2022}\footnote{This theorem is stated in the smooth category in \cite{Menino-Schweitzer2022} but it also works in the topological category just by changing ``smooth'' for ``topological''.}.

\begin{thm}\label{t:C2nonleaf}\cite[Theorem 3.5]{Menino-Schweitzer2022}
Let $W$ be an open manifold with finitely many ends. Assume that
some end of W is (topologically) nonperiodic. Then $W$ cannot be homeomorphic to a proper leaf of a transversely oriented codimension one $C^2$ foliation on a compact manifold.
\end{thm}

\begin{defn}
Let $M_\omega$ be an $n$ dimensional Ghys manifold and let $L_p$ be a closed $n$-dimensional manifold whose fundamental group is isomorphic to $\Z_p$. The manifold $M_{p,\omega}=L_p\# M_\omega$ will be called a {\em perturbed} Ghys manifold.
\end{defn}

\begin{rem}
Observe that any perturbed Ghys manifold can be realized as a leaf of a $C^\infty$ codimension $2$ manifold. In order to see this, let $V_{2}=L_p\# V_0$ and $V_{3} = L_p\# V_1$ that will be called perturbed blocks. Consider {\em generalized} Ghys manifolds $M_\lambda$ for $\lambda\in\{0,1,2,3\}^\Z$ obtained from boundary unions of fundamental and perturbed blocks according with the (bi)sequence $\lambda$. A perturbed Ghys manifold can be seen as a generalized Ghys manifold associated to a sequence $\lambda$ where $2$ or $3$ appears only once. Every generalized Ghys manifold can be realized as a leaf of a codimension $2$ foliation by the same procedure shown in \cite[p. 396]{Ghys1995} which uses the classical realization theorem (relative version) given in \cite{Thurston1974} for foliations of codimension greater than one.
\end{rem}


\begin{lemma}\label{l:rigid_block}
Let $M_{p,\omega}$ be a perturbed Ghys manifold. Let $W_p$ be a manifold with boundary homeomorphic to $L_p$ with a disk removed. If $X_1,X_2$ are submanifolds of $M_{p,\omega}$ homeomorphic to $W_p$ then $X_1\cap X_2\neq\emptyset$.
\end{lemma}
\begin{proof}
The proof is very similar to \cite[Lemme 2.1]{Ghys1985}. In this case $H_1(W_p,\Z)=\Z_p$ and $H_1(M_{p,\omega},\Z)=\Z_p\oplus\bigoplus_{i\in\Z}\Z$. Set $Y = \overline{M_{p,\omega}\setminus (X_1\cup X_2)}$. The Mayer-Vietoris sequence yields a monomorphism from $H_1(X_1\cup X_2,\Z)\oplus H_1(Y,\Z)\to H_1(M_{p,\omega})$ since the boundary of $X_1$ and $X_2$ are spheres of dimension $n-1\geq 2$ (hence trivial $H_1$).

If $X_1\cap X_2=\emptyset$ then $H_1(X_1\cup X_2,\Z)=\Z_p\oplus\Z_p$. This is in contradiction with the fact that $H_1(M_{p,\omega},\Z)$ has only one factor of finite order.
\end{proof}

In the sense of \cite{Menino-Schweitzer2022}, $W_p$ is called a rigid block of $M_{p,\omega}$.

\begin{prop}\label{p:non_nonproper}
For any odd $p$, any $n$-dimensional, $n >3$, perturbed Ghys manifold cannot be homeomorphic to a nonproper leaf of a $C^2$ codimension one foliation on a compact manifold.
\end{prop}
\begin{proof}
Assume that $M_{p,\omega}$ is homeomorphic to some leaf $L$ of a codimension one foliation on a compact manifold. Let $W_p$ be a submanifold (with boundary) of $L$ homeomorphic to $L_p$ with a disk removed. Since $\pi_1(W_p)=\Z_p$, we can apply Reeb stability to $W_p$ and we obtain a transverse negihborhood of $W_p$ in the ambient manifold foliated as a suspension over $W_p$ of a group of homeomorphisms of the interval $[-1,1]$ fixing $0$ ($W_p$ is identifyed with the suspended leaf at $0$). There are no finite groups of $[-1,1]$ with odd order and therefore this suspension must be in fact trivial, i.e., it is foliated as a product.

If $L$ were nonproper, $L$ would meet this neighborhood (foliated as a product) infinitely many times. In particular there would exist infinitely many pairwise disjoint submanifolds homeomorphic to $W_p$ in $L$ and, therefore, in $M_{p,\omega}$, that is in contradiction with Lemma~\ref{l:rigid_block}.
\end{proof}

\begin{cor}
Let $n>3$, any $n$-dimensional forward or backward nonperiodic perturbed Ghys manifold is not homeomorphic to any leaf of any $C^2$ codimension $1$ foliation on a compact manifold.
\end{cor}
\begin{proof}
A perturbed Ghys manifold which is forward or backward nonperiodic satisfies the hypotheses of Theorem~\ref{t:C2nonleaf} and therefore cannot be homeomorphic to a proper leaf of any $C^2$ codimension one foliation on a compact manifold, but Proposition~\ref{p:non_nonproper} implies that it cannot be also homoemorphic to a nonproper leaf. It turns out that it is a $C^2$ codimension one nonleaf.
\end{proof}

Therefore nonperiodic perturbed Ghys manifolds belong to a class of manifolds that lie at the boundary of the realization problem between $C^2$ codimension one and higher codimension foliations, these are the first examples (as far as we know) of manifolds with this property.

\section{Realization of (perturbed) nonperiodic Ghys manifolds}

It is time to study the realization problem for perturbed and nonperiodic Ghys manifolds for $C^r$, $r<2$, codimension one foliations. It will be shown that some perturbed and nonperiodic Ghys manifolds can be realized as leaves in this context. These Ghys manifolds belong to a class of interesting manifolds: those that are $C^2$ codimension one nonleaves but can be realized as leaves of some $C^1$ foliations on some compact manifolds.

It is unclear if every perturbed nonperiodic Ghys manifold satisfies the above property and it would be target of future research, we will state some conjectures about this question in the last section.

For future reference, the {\em saturation} of a set $I$ in a foliation $\FF$, i.e. the set of leaves that meet points in $I$, will be denoted by $\str_\FF(I)$. 


\begin{prop}\label{p:good foliation}
For every $0\leq r < 2$, there exists a $5$-dimensional perturbed and nonperiodic Ghys manifold which is homeomorphic to a (proper) leaf of a $C^r$ codimension one foliation on a compact manifold.
\end{prop}

In order to prove Proposition~\ref{p:good foliation} we need to construct first a suitable foliation with boundary, this is given in the following Lemma. The construction was motivated by the works \cite[Section 9]{Cantwell-Conlon2001} and \cite[Theorem 1.4.8]{Walczak2004}.

\begin{lemma}\label{l:core_foliation}
There exists a $C^\infty$ codimension one foliation $\GG$ on a compact $6$-manifold with boundary such that:
\begin{enumerate}[label=\textbf{A.\arabic*}]
\item \label{A:Property1} The boundary of $\GG$ is transverse and consist of two components, $B_-$ and $B_+$, homeomorphic to $(S^1\times S^3)\times S^1$.

\item \label{A:Property2} The trace foliations on the boundary components are given by products whose leaves have the form $S^1\times S^3\times\{\ast\}$.

\item \label{A:Property3} There exist three pairwise disjoint transverse arcs, $I^-_i$, $i=0,1,2$, in $B_-$ and another three disjoint transverse arcs, $I^+_i$ in $B_+$, $i=0,1,2$, such that
$\str_\GG(I^-_i)=\str_\GG(I^+_i)$, for $i=0,1,2$; $\str_\GG(I^-_i)$ is foliated as a product $V_i\times I$, being $I$ a closed interval where the $V_0,V_1$ are $5$-dimensional fundamental blocks with $\pi_1(V_0)=\Z^3$, $\pi_1(V_1)=\Z^5$ and $V_2$ is a perturbed block.
\end{enumerate}
\end{lemma}

\begin{proof}\
The proof is constructive and separated in several steps.\\

\textbf{Step 1: The initial foliation.}\\

Let $(M_1,\FF_1)$ be a $C^\infty$ codimension one foliation on $M_1 =S^2\times S^1\times [-1,1]$ foliated by leaves homeomorphic to $S^2\times\R$ spiralling to the two tangential boundary components $S^2\times S^1\times\{\pm 1\}$. This foliation is obtained from a Reeb foliation on $S^1\times [-1,1]$, multiplying each leaf by $S^2$.\\

\textbf{Step 2: Turbulizing along transverse circles.}\\

Let us consider three pairs of loops transverse to $\FF_1$. Let us denote them by $\gamma^-_i$ and $\gamma^+_i$, for $i=0,1,2$. Let $D_0^-$ , $D_0^+$, $D_2^-$ and $D_2^+$ be four pairwise disjoint $3$-dimensional disks  transverse  to $\gamma^-_0$, $\gamma^+_0$, $\gamma^-_2$, $\gamma^+_2$ respectively. Similarly, let $T_1^-$ and $T_1^+$ be two pairwise disjoint solid tori (i.e., homeomorphic to $D^2\times S^1$) transverse to $\gamma^-_1$, $\gamma^+_1$ respectively.

\begin{figure}[h]
\begin{center}
\includegraphics[scale=0.2]{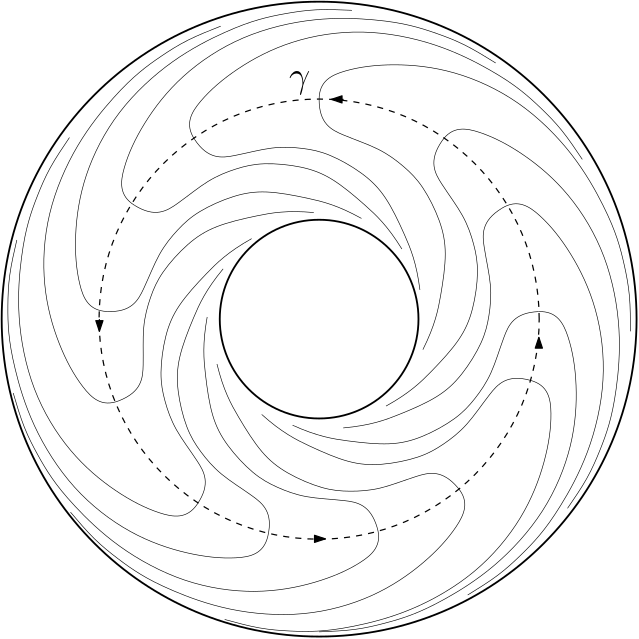}
\includegraphics[scale=0.2]{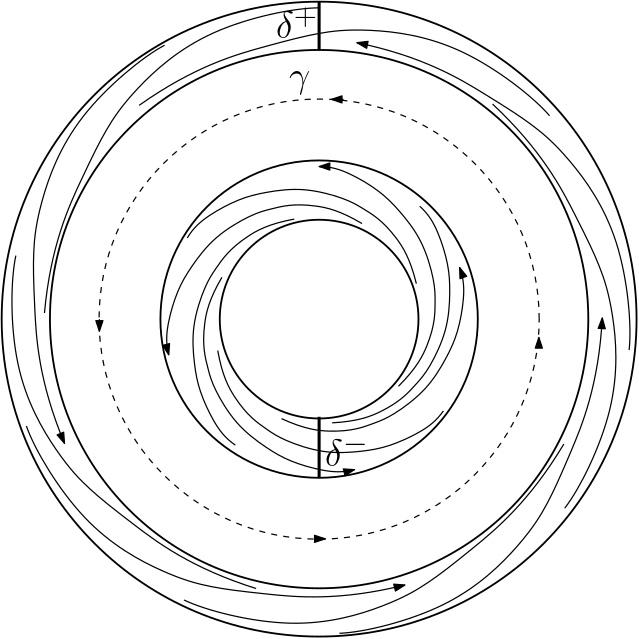}
\caption{Initial foliation (the $S^2$ factor is collapsed) and the turbulized one along a suitably oriented loop $\gamma$. Observe the transverse paths $\delta^{\pm}$ joining the new boundary leaf with the older ones.}
\end{center}
\end{figure}

Turbulize the foliation $\FF_1$ along the loops $\gamma^{\pm}_i$, $i=0,1,2$. We obtain a foliation with three pairs of compact leaves corresponding to the given turbulizations, the first two pairs, corresponding to the turbulizations along the loops $\gamma^{\pm}_0$ and $\gamma^\pm_2$, are four compact leaves homeomorphic to $\partial(D^3)\times S^1 = S^2\times S^1$ and they will be denoted by $C^-_0$, $C^+_0$, $C^-_2$ and $C^+_2$ respectively; the last pair, corresponding with the turbulizations along the loops $\gamma^{\pm}_1$, are homeomorphic to $\partial(D^2\times S^1)\times S^1=T^3$ and will be denoted by $C_1^-$, $C_1^+$.

Let $(M_2,\FF_2)$ be the foliation obtained from the above process and by removing the interior leaves of the generalized Reeb components generated by the previous turbulizations.

Choose the orientations for the turbulizations carefully (see Figure~1) in order to have the following property: There exist six pairs of transverse arcs $\delta^{\pm}_{j,i}:[0,1]\to M_2$, $j=\pm 1$, $i=0,1,2$ such that $\delta^\pm_{- 1,i}(0)\in C^{\pm}_i$ (resp. $\delta^\pm_{+1,i}(0)\in C^\pm_i$) and  $\delta^\pm_{- 1,i}(1)\in S^2\times S^1\times\{-1\}$ (resp. $\delta^\pm_{+1,i}(1)\in S^2\times S^1\times\{+1\}$). 

For future reference, let $\delta^p_{\pm,2}:[0,1]\to M_2$ be another pair of pairwise disjoint transverse arcs (and pairwise disjoint with the defined in the previous steps) that connect $C_2^\pm$ with the boundary components $S^2\times S^1\times\{\pm 1\}$ respectively.

For the sake of readability, the traces of the previous arcs will be noted with bold letters, i.e., $\bm{\delta}$ will denote the trace of the path $\delta$.\\

\textbf{Step 3: Dimensional expansion and perturbation of compact leaves.}\\

Let $(M_3,\FF_3)$ be the foliation of $T^2\times M_2$ whose leaves are obtained from those of $\FF_2$ via cartesian product with $T^2$.

 Remove a tubular neighborhood of $\{\ast\}\times \bm{\delta}^p_{\pm,2}$ and attach to each one of the resulting transverse components a product foliation of the form $W_p\times J$, being $J$ an interval and $W_p$ a manifold with boundary obtained by removing a disk from a $5$-dimensional closed manifold $L_p$ with $\pi_1(L_p)=\Z_p$, $p$ odd. This new foliation will be denoted by $(M_4,\FF_4)$.
 
There are four pairs of compact leaves in $\FF_4$: $T^2\times C_0^\pm$, $T^2\times C_1^\pm$, $L_p\# (T^2\times C_2^\pm)$ and $L_p\# (T^2\times S^2\times S^1\times\{\pm 1\})$.\\

\textbf{Step 4: Tunneling.}\\

Let $a$ and $b$ be two simple loops in $T^2$ whose homotopy classes generate $\pi_1(T^2)$. Let $2b$ a simple loop in $T^2$ whose homotopy class agrees with $2[b]$. 

Let $N^-_{a,i}$, $N^-_{2b,i}$, $N^+_{a,i}$, $N^+_{2b,i}$ be pairwise disjoint open tubular neighborhoods of $\bm{a}\times \bm{\delta}^-_{-1,i}, \bm{2b}\times \bm{\delta}^-_{+1,i}, \bm{a}\times \bm{\delta}^+_{+1,i}$, $\bm{2b}\times \bm{\delta}^+_{-1,i}$, respectively, for $i=0,1,2$.


Let $(M_5,\FF_5)$ be the foliation obtained from $(M_4,\FF_4)$ by removing the above neighborhoods from $M_4$. The result of this tunneling process is the apparition of twelve transverse boundary components on $M_4$ which are homeomorphic to $S^3\times S^1\times \bm{\delta}^\star_{\ast 1,i}$, $i=0,1,2$ and $\ast,\star\in\{-,+\}$. The foliation $\FF_5$ has eight tangential boundary leaves:
\begin{itemize}
\item Two boundary leaves come from the leaves $L_p\#(T^2\times S^2\times S^1\times\{\pm 1\})$ of $\FF_4$, but now  each one of these leaves have six boundary components corresponding to the boundary of tubular neighborhoods of the loops $\bm{a}\times \{\delta^-_{- 1,i}(1)\}$, $\bm{2b}\times \delta^+_{-1,i}(1)$, $i=0,1,2$ in $L_p\#(T^2\times S^2\times S^1\times\{- 1\})$ and $\bm{a}\times \{\delta^+_{ +1,i}(1)\}$, $\bm{2b}\times \delta^-_{+1,i}(1)$, $i=0,1,2$, in $L_p\#(T^2\times S^2\times S^1\times\{+ 1\})$. These two boundary leaves will be noted as $K_-$ and $K_+$.

\item Three pairs of compact boundary leaves come from the leaves of $\FF_4$ denoted as $T^2\times C^-_i$ and $T^2\times C^+_i$  for $i = 0,1$ and $L_p\# (T_2\times C^-_2)$, $L_p\# (T_2\times C^+_2)$. Each one of these compact leaves have two boundary components corresponding to the boundary of tubular neighborhoods of the loops $\bm{a}\times \delta^-_{-1,i}(0)$, $\bm{2b}\times \{\delta^-_{+1,i}(0)\}$, for the case of $C^-_i$, and $\bm{a}\times \{\delta^+_{+1,i}(0)\}$, $\bm{2b}\times \{\delta^+_{-1,i}(0)\}$ for the case of $C^+_i$, $i=0,1,2$. These pairs of boundary leaves of $\FF_4$ will be noted as $V_i^-$ and $V_i^+$ for $i=0,1,2$.
\end{itemize}
Observe that $V_i^-$ and $V^+_i$  are homeomorphic for $i = 0,1,2$ and $\pi_1(V^\pm_0)=\Z^3$, $\pi_1(V_1^\pm)=\Z^5$, $\pi_1(V_2^\pm) = \Z_p\ast\Z_3$. It follows that $V_0^\pm, V_1^\pm$ are homeomorphic to fundamental blocks $V_0$ and $V_1$ of rank $3$ and $5$ respectively. In a similar way $V_2^\pm$ are homeomorphic to $L_p\# V_0^\pm$ and therefore these are homeomorphic to a perturbed block $V_2$ of rank $3$.\\

\textbf{Step 5: The bridges}\\

Let $\GG_i$ be the foliation of $V_i\times [-1,1]$ foliated as a product, $i=0,1,2$. Let $(M_6,\FF_6)$ be the foliation obtained from $\FF_5$ by attaching the foliations $\GG_i$ via the identification of tangential boundary leaves $V_i^-$ with $V_i\times\{-1\}$ and $V_i^+$ with $V_i\times\{+1\}$ for $i=0,1,2$. The foliation $\FF_6$ has two tangential boundary leaves ($K_-$ and $K_+$) that can be also attached, the homeomorphism used to attach these boundary leaves will identify boundary components as follows.

Let $\partial K^+_{a,i}$, $\partial K^+_{2b,i}$, be the boundary components of $K^+$ which correspond, respectively, with the boundaries of the removed tubular neighborhoods of $\bm{a}\times\{\delta^+_{+1,i}(1)\}$, $\bm{2b}\times\{\delta^-_{+1,i}(1)\}$, for $i=1,2,3$. 

Simlarly define $\partial K^-_{a,i}$, $\partial K^-_{2b,i}$, $i=0,1,2$, as the respective boundary components of $K_-$ which correspond, respectively, with the boundaries of the tubular neighborhoods of $\bm{a}\times\{\delta^-_{-1,i}(1)\}$, $\bm{2b}\times\{\delta^+_{-1,i}(1)\}$, $i=0,1,2$.

Let $f:K_-\to K_+$ be a homeomorphism that satisfies $f(\partial K^+_{a,0}) = \partial K^-_{a,1}$, $f(\partial K^+_{a,1}) = \partial K^-_{a,2}$, $f(\partial K^+_{a,2}) = \partial K^-_{a,0}$, $f(\partial K^-_{2b,0}) = \partial K^+_{2b,1}$, $f(\partial K^-_{2b,1}) = \partial K^+_{2b,2}$ and $f(\partial K^-_{2b,2}) = \partial K^+_{2b,0}$.

Let $(M,\GG)$ be the foliation obtained from $(M_5,\FF_5)$ by identifying the tangential boundary leaves $K_-$ and $K_+$ via the homeomorphism $f$. The resulting foliation has two transverse boundary components that satisfy the conditions of Lemma~\ref{l:core_foliation}. 



In order to see this, let $x,y\in T^2$ different points. Set $x^j_i = (x,\delta^j_{j,i}(0))$ and $y^j_i = (y,\delta^j_{-j,i}(1))$ for $i\in\{0,1,2\}$ and $j\in\{-,+\}$ (check Figure~2). These points belonged to $M_3$ but without loss of generality they can be assumed to belong to $M_5$. Moreover, assume that $x^\pm_i, y^\pm_i\in \partial V_i^\pm$ for $i=0,1,2$, where $x^\pm_i$ belong to the boundary component obtained by removing a tubular neighborhood of $\bm{a}$ and $y^\pm_i$ to the other component. Let $x_i, y_i\in V_i$ denote the points in $V_i$ so that $x_i^\pm$, $y_i^\pm$ are attached to $(x_i,\pm 1)$ and $(y_i,\pm 1)$ respectively in the bridge foliations.


Let us consider the subindices $\{0,1,2\}\equiv\Z_3$, hence $2+1 = 0 (\text{mod} 3)$. Without loss of generality, assume that $f((x,\delta^j_{j,i}(1)))=(x,\delta^{-j}_{-j,i+1}(1))$ and $f((y,\delta^j_{-j,i}(1)))=(y,\delta^{-j}_{j,i+1}(1))$ for $i\in\Z_3$ and $j\in\{-,+\}$. Let us also denote the arcs $\{x_i\}\times [-1,1]$ in $\GG_i$ by $I_i^-$, $i=0,1,2$, and, similarly, the arcs $\{y_i\}\times [-1,1]$ will be denoted as $I_i^+$ for $i=0,1,2$ respectively.

There are two transverse boundary components in $\GG$, let us call them $B_-$ and $B_+$ . Each one of these componentes have a transverse circle, denoted by $S_-$ and $S_+$ respectively, these are obtained by attaching the transverse arcs: 
\begin{align*}
S_- &=(\{x\}\times\bm{\delta}^-_{-1,0})\cup I_0^- \cup (\{x\}\times\bm{\delta}^+_{+1,0}) \cup (\{x\}\times \bm{\delta}^-_{-1,1}) \cup\\
& I_1^- \cup
 (\{x\}\times \bm{\delta}^+_{+1,1})\cup (\{x\}\times \bm{\delta}^-_{-1,2})\cup I_2^-\cup (\{x\}\times\bm{\delta}^+_{+1,2})\;,
\end{align*}
\begin{align*}
S_+ &=(\{y\}\times\bm{\delta}^-_{+1,0})\cup I_0^+ \cup (\{y\}\times\bm{\delta}^+_{-1,0}) \cup (\{y\}\times \delta^-_{+1,1}) \cup\\
& I_1^+ \cup
 (\{y\}\times \bm{\delta}^+_{-1,1})\cup (\{y\}\times \bm{\delta}^-_{+1,2})\cup I_2^+\cup (\{y\}\times\bm{\delta}^+_{-1,2})\;,
\end{align*}

The transverse arcs $\{x_i\}\times [-1,1]$ (resp. $\{y_i\}\times [-1,1]$) are complete transversals of the bridge foliations $\GG_0$, $\GG_1$ and $\GG_2$ respectively and they play the role of $I_i^-$ (resp. $I^+_i$) in the estatement of Lemma~\ref{l:core_foliation}.

\end{proof}

\begin{figure}[h]
\begin{center}
\includegraphics[scale=0.3]{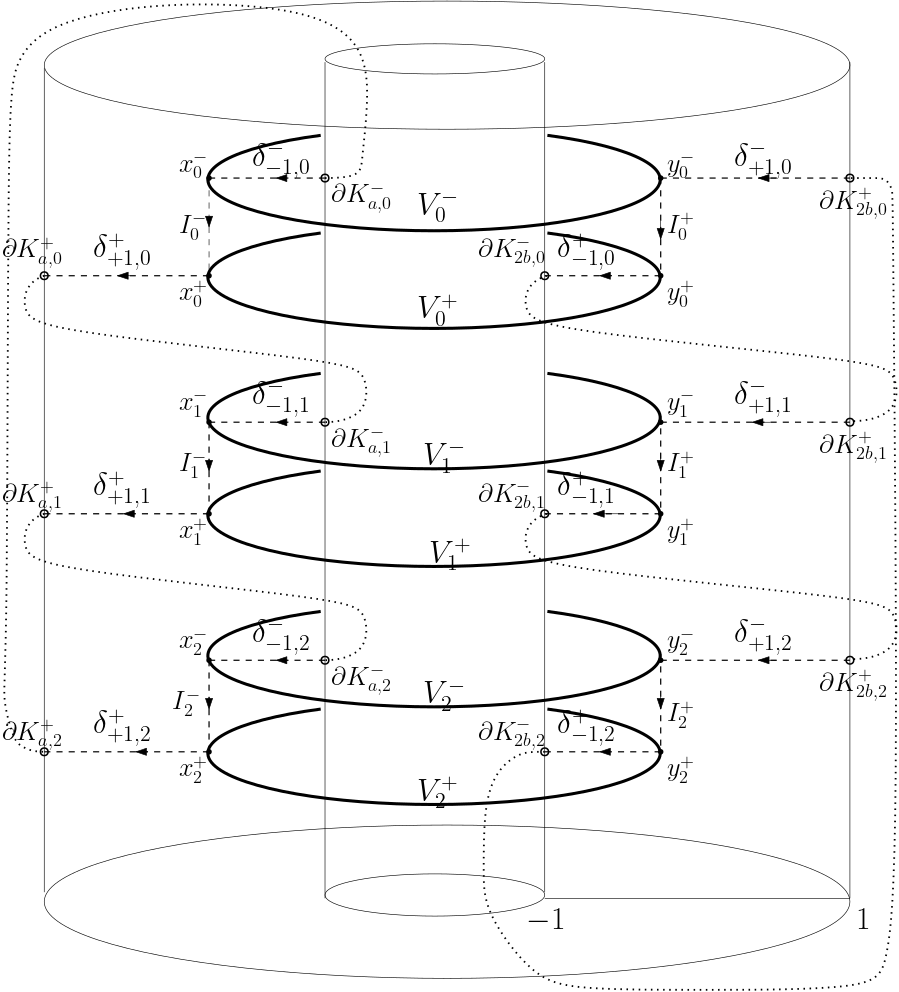}
\caption{The elements of the foliation $\FF_6$ (the vertical dimension must be interpreted as a $T^2\times S^2$ factor at this step). Bold circles represent closed leaves obtained by turbulization and tunneling. The dotted arcs join boundary components of the exterior compact leaves (external and internal cylinders) indentified by the homeomorphism $f$. The transverse circles $S_-$ and $S_+$ can be seen depicted to the right and to the left respectively (dashed lines).}
\end{center}
\end{figure}

\begin{proof}[Proof of Proposition~\ref{p:good foliation}]
Let us consider the foliation $\GG$ given by Lemma~\ref{l:core_foliation}, let $B_-$ and $B_+$ be its transverse boundary components and let $S_-$ and $S_+$ be complete transversals of the trace foliations on $B_-$ and $B_+$. Choose parametrizations of $S_-$ and $S_+$ that identify $I_i^\pm\subset S^\pm$ with the same interval $I_i\subset S^1$, for $i=0,1,2$. Let $\varphi:S^1\to S^1$ be a Denjoy diffeomorphism (that can be chosen of regularity $C^{1+\alpha}$ for any $\alpha\in [0,1)$, see \cite{Herman,Navas}) such that its minimal set $\Lambda$ is included in the interior of $I_0\cup I_1$ and, moreover, $\Lambda\cap I_0\neq \emptyset\neq \Lambda\cap I_1$. We can also assume, for the sake of simplicity, that $\varphi$ acts transitively on the gaps of $\Lambda$. 

Under the above condition, the complement of $I_0\cup I_1$ is included in just two gaps of $\Lambda$, one of the gaps includes $I_2$ and there exists $k\in\Z^\ast$ such that $\varphi^k(I_2)$ is included in the other gap. Let us assume that the extreme points of $\varphi^k(I_2)$ belong to $I_0\cup I_1$. This extra condition is easy to achieve from a Denjoy diffeomorphism by applying a suitable conjugation if necessary, it will be referred as the ``gap condition''.

Let $\FF_\varphi$ be the foliation obtained from $\GG$ by identifying the transverse boundary components  under the homeomorphism $$\Phi: B_-\equiv (S^3\times S^1)\times S_- \to B_+\equiv (S^3\times S^1)\times S_+\;,\ \Phi(x,z)=(x,\varphi(z))\;.$$

The saturation of $\Lambda$ in $\FF_\varphi$ defines an exceptional minimal set and their leaves are Ghys manifolds. This follows from the fact that $\Lambda\subset I_0\cup I_1$ and the saturations of the arcs $I_0$ and $I_1$ consists of the bridge foliations $\GG_0$ and $\GG_1$ that are product foliations with the Ghys blocks $V_0$ and $V_1$ as leaves respectively. The ataching map $\Phi$ therefore attachs a fundamental block with another following a $\varphi$-orbit in $\Lambda$, the Ghys block will be $V_0$ when the point belongs to $I_0$ and $V_1$ if it belongs to $I_1$.

Since $\Lambda\cap I_0\neq \emptyset\neq \Lambda\cap I_1$, every leaf meeting $\Lambda$ is, in fact, a nonperiodic Ghys manifold. An easy way to see this is assuming, {\em ad absurdum}, that some of these leaves is a periodic Ghys manifold. If this were the case, then it should exist some $N\in\N$, $x\in\Lambda$ and $\xi\in\{0,1\}$ such that  $\varphi^{N\cdot n}(x)\in I_{\xi}$ for all $n\in\Z$. But $\varphi^N$ has the same minimal set as $\varphi$, this is consequence of the fact that any irrational rotation is dense in the circle and $\varphi$ is semiconjugated with an irrational rotation of the circle, therefore the $\varphi^N$-orbit of $x$ must meet both $I_0$ and $I_1$ and we reach a contradiction.

Recall that $\varphi$ acts transitively in the gaps of $\Lambda$, the $\varphi$-orbit of a point in a gap of $\Lambda$ meets each gap just once. Let $x$ be an extreme point of $I_2$. The ``gap condition'' on $\varphi$ implies that $\varphi^n(x)\in I_0\cup I_1$ for all $n\neq 0$, therefore the leaf passing by $x$ is obtained as a union of Ghys blocks $V_0$ and $V_1$ and exactly one perturbed block $V_2$. This leaf is a perturbed Ghys manifold which is not forward nor backward periodic.
\end{proof}

\section{Repetitive leaves}

Our previous construction is rather flexible and hence it is possible to realize lots of different (perturbed) nonperiodic Ghys manifolds as leaves of codimension one foliations. In this section we show that this construction always leads to the apparition of nonperiodic leaves which are (forward and backward) repetitive.

Repetitiveness is a useful concept that comes from tiling theory. In our context we are only interested in the case of (bi)sequences $\omega\in\{0,1\}^\Z$. A {\em word} of a (bi)sequence is a finite and ordered list of consecutive elements, a $k$-word is a word with $k$ elements.

\begin{defn}[Repetitive (bi)sequences]
A (bi)sequence is called repetitive if for every $k\in\N$ there exists $\ell_k\in\N$ such that any $k$-word of the sequence is contained in every $\ell_k$-word.
\end{defn}

\begin{defn}[Forward and backward repetitive (bi)sequence]
A sequence $\omega\in\{0,1\}^\Z$ is called forward (resp. backward) repetitive if there exists $N>0$ such that the sequences $(\omega(n))_{n>N}$ and $(\omega(n))_{n<-N}$ are both repetitive.
\end{defn}

The next Proposition is a consequence of a well known fact on sturmian sequences. We include a proof for completeness.

\begin{prop}\label{p:repetitive leaves}
For any choice of the Denjoy homeomorphism $\varphi$ the leaves of $\FF_\varphi$ that meets the minimal set $\Lambda$ are (forward and backward) nonperiodic Ghys manifolds which are repetitive.
\end{prop}

\begin{proof}
The sequences associated to the leaves that meet $\Lambda$ are examples of sequences obtained by a coding under a rotation. Since the rotation number is irrational, the sequences are recurrent sturmian sequences (see \cite{Morse-Hedlund}) and it is well known that they are forward and backward repetitive.
\end{proof}

\begin{cor}\label{c:almost periodic proper leaves}
For any choice of the Denjoy homeomorphism $\varphi$, the proper leaves of $\FF^p_\varphi$ which are (perturbed) Ghys manifolds are forward and backward repetitive.
\end{cor}
\begin{proof}
The perturbed Ghys manifold which is realized as a proper leaf has an associated sequence that is in correspondence with an associated sequence to some leaf that meets the minimal set, with the possibly difference that at some point a Ghys block is substituted with a perturbed Ghys block (not necessarily associated to the same simbol $0$ or $1$). The result follows from Proposition~\ref{p:repetitive leaves}.
\end{proof}

\subsection{Realizable sequences and obstructions}

It is not difficult to make an explicit description of the repetitive sequences obtained by our construction, this ultimately depends on the continuous fraction expansion of $\theta$ and the relative position of $\Lambda$ respect to $I_0$ and $I_1$ (more precisely, in the pair of gaps that are not covered by $I_0\cup I_1$). 

If the pair of gaps $J_0,J_1$ which are not included in $I_0\cup I_1$ satisfy $\varphi(J_0)=J_1$, then the sequences realized in our construction are Sturmian sequences. The $n$-complexity of a sequence $\omega$ is the number of different $n$-words in the sequence. A sequence $\omega$ is periodic if and only if its $n$-complexity $\leq n$ for all $n$. A sequence is Sturmian if and only if it is recurrent and its $n$-complexity is $n+1$ for all $n$. Every Sturmian sequence can be obtained by a coding given by an irrational rotation $R_\theta$, with rotation number $\theta$, over the partitions $[0,\theta), [\theta,1)$ or $(0,\theta], (\theta,1]$ of $\R/\Z$ \cite{Morse-Hedlund}, and these are precisely the sequences obtained (up to shift or symbol inversion) when $\varphi(J_0)=J_1$. 

More mileage can be obtained from this relation. The next proposition solves, in the positive, an open question given in \cite[pp. 399]{Ghys1995}.

\begin{prop}
Let $\varphi:S^1\to S^1$ be a Denjoy homeomorphism and let $\Lambda$ be its exceptional minimal set. Then the leaves of the minimal set of the foliation $\FF_\varphi$ are pairwise nonhomoemorphic.
\end{prop}
\begin{proof}
The coding of the sequences associated to the Ghys manifolds that are leaves in $\str(\Lambda)$ induces a conjugation between de Denjoy system $(\Lambda,\varphi_\Lambda)$ and a sturmian sufshift (as it is done in \cite[Proposition 3]{Masui}). Henceforth, the leaves of the minimal set of the foliation $\FF_\varphi$ are pairwise nonhomeomorphic since they are Ghys manifolds associated to sequences in different orbits by the shift action.
\end{proof}

When $\varphi^k(J_0)=J_1$ for $k\neq \pm 1$ then the sequence obtained does not need to be sturmian but still can be described in terms of the difference of two sturmian sequences (associated to the rotations $R_\theta$ and $R_{k\theta}$) \cite{Hubert}.

The repetitive (bi)sequences that can be realized as associated sequences of a Ghys leaf of some $\FF_\varphi$ have some natural restrictions. Recall from the proof of  Proposition~\ref{p:repetitive leaves} that the sequence associated to the realized Ghys manifold depends on the cardinal function $\chi$ over a semiopen arc $\Lambda_1$ applied through an orbit of an irrational rotation. Since irrational rotations induce an uniquely ergodic system it follows that the Birkhoff sums must converge to the Lebesgue measure of this interval. The extreme point of the interval must belong, by construction, to the same orbit under the irrational rotation, thus $|\Lambda_1| = k\cdot\theta (\text{mod}\ 1)\ \text{or}\ 1-k\cdot\theta (\text{mod}\ 1)$ for some $k\in\Z$.

Thus, if $\omega_\varphi$ is the (bi)sequence associated to a Ghys manifold realized as a leaf of $\FF_\varphi$, then it  must satisfy that:
$$
\lim\limits_{n\to\pm\infty}\dfrac{1}{n}\sum_{i=0}^n \omega_\varphi(i) = k\cdot\theta (\text{mod}\ 1)\ \text{or}\ 1-k\cdot\theta (\text{mod}\ 1)
$$

The previous sum can be also interpreted as the frequency of the symbol $1$ in the sequence. It is not hard to find nonperiodic repetitive sequences where the average of symbols converge to a rational number, for instance, the frequency of symbols in the Thue-Morse sequence converge to $1/2$. Moreover, there are repetitive sequences where the frequency of a symbol is not well defined. This kind of repetitive sequences cannot be associated to Ghys manifolds realized by our construction. Moreover, a (bi)sequence that is forward repetitive, forward nonperiodic but backward periodic cannot be handled by our construction.

\subsection{Simple modifications}

The familiy of sequences associated to Ghys manifolds realizable as leaves can be increased if we allow some simple modifications in our original construction. In the next three remarks we birefly explain three different kinds of modifications.

\begin{rem}[Concatenation of foliations]
Let $\GG$ be the foliation defined in Lemma 3.2. And let $\GG'$ be a copy of that foliation. Let $S_-',S_+'$ be the corresponding transverse circles in $\GG'$. Let $I_0',I_1'\in S^1$ be the corresponding intervals in a common parametrization of $S_-$ and $S_+$ whose saturations are product foliations with the fundamental blocks $V_0$ and $V_1$.

Up to a circle reparametrization $I_0'$ and $I_1'$ can be identified with arbitrary disjoint intervals in $S^1$. Let $\varphi:S^1\to S^1$ be a $C^r$ Denjoy diffeomorphism and $\Lambda$ its minimal set. A {\em concatenation} of $\GG$ and $\GG'$ is the foliation obtained by a transverse gluing of $\GG$ with $\GG'$ so that $S_+$ is identified with $S_-'$ via the identity map (relative to the given circle parametrizations) and $S_+$ is identified with $S_-'$ via $\varphi$. The concatenation is called {\em coherent} if $\Lambda\subset (I_0\cup I_1)\cap (I_0'\cup I_1')$.

The leaves of the minimal set in a coherent concatenation are still Ghys manifolds. Set $I_{ij}=I_i\cap I_j'$ for $i,j\in\{0,1\}$. The associated sequence of a leaf passing through a point $x\in\Lambda$ depends in the $\varphi$-orbit of $x$: the word $ij$ is added to the sequence whenever $\varphi^k(x)\in I_{ij}$. These sequences are generalizations of Sturmian sequences. In fact if the words $ij$, $i,j\in\{0,1\}$, are substituted by distinct symbols, then the sequence is a special kind of a coding by rotation, these were studied in \cite{Didier}. 
\end{rem}

\begin{rem}[Adding Bridges]
More bridges like $\GG_0$ and/or $\GG_1$ can be added to the construction. This allows to obtain more general sequences. More precisely: let $x_1,\dots,x_m$ be points in the $R_\theta$ orbit of $0\in\R/\Z$, set $0=x_0$ and $x_{m+1}=1$. Let $\xi:\{0,\dots,m\}\to\{0,1\}$ be an arbitrary function, let us consider the sequence $\overline{\omega}(x)= \xi(k)$ (resp. $\underline{\omega}(x)=\xi(k)$) for the unique $k$ such that $x\in [x_k,x_{k+1})$ (resp. $x\in (x_k,x_{k+1}]$).

Any Ghys manifold whose associated sequence is coded as before can be realized as the leaf of a $C^r$ codimension one foliation on a closed $6$-manifold.
\end{rem}

\begin{rem}[Adding topological symbols]
The symbols in the associated sequence of a Ghys manifold reflect the sequential change in the topology of the manifold. The topological symbols are the ranks of the fundamental groups of the fundamental blocks. Thus, the Ghys manifolds can be defined for sequences over more symbols whenever more fundamental blocks with different ranks are available.

In our realization, the fundamental blocks ultimately appear as boundary leaves of turbulizations over neighborhoods whith different topologies. If we want to realize some nonperiodic Ghys-manifold over $3$-symbols in a codimension one foliation, we can consider a similar construction but increasing the leaf dimension. In this case consider the initial foliation $\FF_0$ as a product of the one dimensional Reeb foliation with $S^3$ (instead of $S^2$). A local section to a transverse loop $\gamma$ is homeomorphic to $\R^4$. Since both $D^3\times S^1$ and $D^2\times T^2$ can be embedded in $\R^4$ it follows that we can make turbulizations along product neighborhoods of $\gamma$ homeomorphic to $D^4\times S^1$, $D^3\times S^1\times S^1$ or $D^2\times T^2\times S^1$. The boundary leaves of the respective generalized Reeb components are $S^3\times S^1$, $S^2\times T^2$ and $T^4$. After the dimensional expansion and the deletion of neighborhoods of the loops $a$ and $2b$ we obtain fundamental blocks of rank $3$, $4$ and $6$ that can be used to describe nonperiodic Ghys manifolds on $3$ symbols as desired.

More symbols can be added at the cost of increasing the leaf dimension of the foliation. The dimensional cost on the leaf topology can be an artifact in our construction but it is unclear how to improve it. Reduce the leaf dimension to $4$ is out of the reach at the present state of the art.
\end{rem}

Of course, the previous modifications can be mixed together in orther to obtain more general repetitive sequences. In any case, all the obtained sequences are, essentially, coding from rotations over finitely many intervals. This is, of course, produced by the particularities of the given foliation $\FF_\varphi$, but it seems very possible to exploit the construction in other ways in order to obtain more general sequences.

\begin{quest}
Is it possible to realize every nonperiodic repetitive Ghys manifold as a leaf of a codimension one foliation?
\end{quest}

\subsection{Topological repetitiveness}

\begin{rem}
Recall that there exists a notion of (end) repetitiveness for Riemannian manifolds (that generalizes the notion of repetitiveness in tilings), moreover, it is known that the leaves of minimal foliated compact spaces without holonomy are repetitive from this geometric point of view (see \cite{Alvarez-Candel2018}), in some sense Proposition~\ref{p:repetitive leaves} can be seen as a corollary of this result. 
\end{rem}

Since we are working in the topological category we are interested in a topological notion of repetitiveness that generalizes the concept of topological periodicity and includes the geometric repetitiveness as a particular case.

\begin{defn}[Topologically repetitive end]
An end $\mathbf{e}$ of a manifold $M$ is called {\em topologically repetitive} if there exists a neighborhood $N$ of $\mathbf{e}$ such that for any compact set $K\subset N$ and every neighborhood system $\{N_n\}_{n\in\N}$ of $\mathbf{e}$ there exist (pairwise disjoint) subsets $K_n\subset N_n$ homeomorphic to $K$.
\end{defn}

In a topologically repetitive end, every compact set sufficiently close to that end must appear infinitely many times accumulating to that end. Therefore a topologically repetitive end implies that there no exists rigid blocks (in the sense of \cite{Menino-Schweitzer2022}) approaching that end. The approximation of ends by rigid blocks is the main tool to produce $C^0$ nonleaves of codimension one foliations and therefore it can be conjectured that this could be a necessary condition for (proper) leaves of codimension one foliations.

\begin{quest}
Does there exist a codimension one foliation on a compact manifold having a proper leaf with finitely many ends such that some of its ends is not topologically repetitive?
\end{quest}

Observe that if the answer to the previous question were negative then the perturbed Ghys manifolds whose associated sequence is not forward or backward repetitive would be nonleaves of codimension one foliations (with no regularity assumptions) that can be realized as leaves in some codimension two foliations.

Recall that every open surface can be realized as a (nonproper) leaf of a $C^\infty$ codimension one foliation \cite{Cantwell-Conlon} and there exist surfaces with non topologically repetitive ends, all of them with infinitely many ends. Therefore the previous question is only interesting for proper leaves with finitely many ends. It is also unknown for the author the existence of an example of a codimension one foliation on a compact manifold having a nonproper leaf with finitely many ends and at least one topologically nonrepetitive end.

\end{document}